\documentclass[a4paper,11pt]{amsproc}

\usepackage{xargs}
\usepackage{mathtools, todonotes}
\usepackage{amsmath,amssymb,amsfonts,amsthm}

\usepackage[colorlinks]{hyperref}
\usepackage{amsrefs}
\usepackage{enumerate, tikz-cd, nicefrac}

\usepackage{breqn}

\title{On finite approximations of transitive graphs}

\author{Andreas Thom}
\address{Andreas Thom, TU Dresden, 01062 Dresden} 
\email{andreas.thom@tu-dresden.de}

\theoremstyle{plain}
\newtheorem{theorem}{Theorem}

\newtheorem{lemma}[theorem]{Lemma}
\newtheorem{corollary}[theorem]{Corollary}

\newtheorem{question}[theorem]{Question}

\theoremstyle{definition}
\newtheorem{remark}[theorem]{Remark}

\newcommand{\beq}{\begin{equation}}
\newcommand{\eeq}{\end{equation}}
\newcommand{\beqn}{\begin{equation*}}
\newcommand{\eeqn}{\end{equation*}}
\newcommand{\brq}{\begin{dmath}[compact]}
\newcommand{\erq}{\end{dmath}}
\newcommand{\brqn}{\begin{dmath*}[compact]}
\newcommand{\erqn}{\end{dmath*}}
\newcommand{\bag}{\begin{align}}
\newcommand{\eag}{\end{align}}
\newcommand{\bagn}{\begin{align*}}
\newcommand{\eagn}{\end{align*}}

\newcommand{\vertiii}[1]{{|\kern-0.2ex|\kern-0.2ex| #1 
    |\kern-0.2ex|\kern-0.2ex|}}

\begin{document}

\begin{abstract}
In this  note we answer a question of Johannes Carmesin, which was circulated at the Oberwolfach Workshop on "Graph Theory" in January 2025. We provide a unimodular, locally finite, and vertex-transitive graph without any perfect finite $r$-local model for $r \in \mathbb N$ large enough.
\end{abstract}

\maketitle

Let $G=(V,E)$ be a graph, by which we mean a set $V$ of vertices and a collection $E$ of $2$-element subsets of $V$ called edges. We say that $G$ is \emph{vertex-transitive}, if its automorphism group is acting transitively on the set $V$. We say that $G$ is \emph{automorphism-regular} if every automorphism fixing a vertex of $G$ acts trivially on $G$, i.e., if the action of the automorphism group of $G$ on the vertex set is regular.

We have the usual distance function $d \colon V \times V \to \mathbb N \cup \{\infty\}$ which assigns to the pair $(v,w) \in V^2$ the length of the shortest path from $v$ to $w$. We denote by $B_G(v,r)$ the induced rooted subgraph on the set of vertices at distance at most $r$ from $v \in G$, rooted at $v$; this is called the ball around $v$ of radius $r$. We call a graph \emph{locally finite} if the ball of radius $r$ around every vertex is a finite graph. We say that a locally finite vertex-transitive graph $G$ has a \emph{perfect finite $r$-local model} if there exists a finite graph $G_0$, such that the ball of radius $r$ around any vertex of $G_0$ is rooted isomorphic to the ball of radius $r$ around some vertex in $G$.

\begin{question}[Carmesin] \label{carmesin}
Does every uni-modular, locally finite, and vertex-transitive graph $G$ have a perfect finite $r$-local model for any $r \in \mathbb N.$
\end{question}

We show in Corollary \ref{answer} that the answer to this question is negative.

\begin{remark}
Note that it is a folklore observation that the grandfather graph (which is a directed graph) cannot be approximated in the above sense by finite directed graphs. Encoding the directions into the graph structure in a standard way yields a locally finite and vertex-transitive graph without any perfect finite $2$-local model, see \cite{MR3460154}. However, the resulting graph is not unimodular and it is natural to add unimodularity to the assumptions when asking the above question.
\end{remark}

A common source of unimodular, locally finite, and vertex-transitive graphs are Cayley graphs of finitely generated groups. Let $\Gamma$ be a finitely generated group and $S$ a generating set, where we assume that $1 \not\in S$. The Cayley graph ${\rm Cay}(\Gamma,S)$ of $G$ with respect to $S$ is a graph with vertex set $\Gamma$ and edge set consisting of all pairs of the form $\{g,gs\}$ with $g \in \Gamma, s \in S.$

\begin{lemma}
Let $\Gamma$ be a finitely generated group with a finite generating set $S=S^{-1}$ with $1 \not \in S$, such that ${\rm Cay}(\Gamma,S)$ is automorphism-regular. For every $r \in \mathbb N$, there exists $r_0:=r_0(r) \in \mathbb N$, such that any rooted automorphism of $B_G(e,r_0)$ fixes $B_G(e,r)$ pointwise.
\end{lemma}
\begin{proof}
By contradition, there exists $r \in \mathbb N$, such that no $r_0$ can be found with the property. Then, there exists a sequence of rooted automorphisms $\varphi_n$ of $B_G(1,n)$, for $n \geq r$, such that each has a non-trivial restriction to $B_G(r)$. We may choose a subsequence of $(\varphi_n)_n$ which converges pointwise to a non-trivial automorphism of ${\rm Cay}(\Gamma,S)$ that fixes $e\in \Gamma.$ This contradicts the automorphism-regularity of ${\rm Cay}(\Gamma,S)$.
\end{proof}

For every subgroup $\Lambda \leq \Gamma$, we consider the so-called Schreier graph ${\rm Sch}(\Lambda \backslash \Gamma,S)$ with vertex set $\Lambda \backslash \Gamma$ and edges of the form $\{\Lambda g,\Lambda gs \}$ for $g \in \Gamma, s \in S.$ More generally, such a graph can be associated to any right $\Gamma$-action on a set $X.$

\begin{theorem} \label{thm:main}
Let $\Gamma$ be a finitely presented group with a finite generating set $S=S^{-1}$ with $1 \not \in S$, such that ${\rm Cay}(\Gamma,S)$ is automorphism-regular. There exists $r \geq 1$, such that each connected perfect finite $r$-local model arises as the Schreier graph of a finite index subgroup $\Lambda \leq \Gamma$.
\end{theorem}
\begin{proof} We set $G:={\rm Cay}(\Gamma,S)$. Assume that $G_0=(V,E)$ is a perfect finite $r$-local model for $G$ for some $r \in \mathbb N$ to be determined later.

Assume that $r \geq r_0(2)$ from the previous lemma. Since there exists a unique rooted isomorphism between $B_{G_0}(v,1)$ and $B_G(e,1)$ that extends to a rooted isomorphism between $B_{G_0}(v,r_0)$ and $B_G(e,r_0)$, every edge $\{v,w\}$ determines a unique edge $\{e,s\}$ with $s \in S$ in $B_G(e,1)$; so that we can label the directed edge $(v,w)$ by $s \in S$. Because this rooted isomorphism extends uniquely to an isomorphism between $B_{G_0}(v,2)$ and $B_G(e,2)$, the inverse edge $(w,v)$ gets the label $s^{-1}$ and this allows us to define a transitive right-action of the free group generated by $S$ on $V$. If $r \geq r_0(r')$, using the previous lemma, so that $B_G(e,r')$ includes all relations of a finite presentation of $\Gamma$, then this action factors through $\Gamma.$ We may now set $\Lambda := {\rm Stab}_\Gamma(v) \leq \Gamma$ for some $v \in V$ and identify $G_0$ with the Schreier graph for the right $\Gamma$-action on $\Lambda \backslash \Gamma.$
\end{proof}

\begin{remark}
Note that the previous theorem is not true without the assumption of automorphism-regularity of the Cayley graph. Indeed, for $\Gamma = \mathbb Z^2$ with the usual generating set, one may consider suitable discretizations of the Klein bottle that do provide perfect finite $r$-local models for $r$ arbitrarily large and yet (even though vertex-transitive) do not come finite index subgroups of $\mathbb  Z^2$  -- but they do come from a finite index subgroup of ${\rm Aut}({\rm Cay}(\mathbb Z^2,S)).$ 
\end{remark}

\begin{corollary} \label{answer}
The answer to Question \ref{carmesin} is negative.
\end{corollary}
\begin{proof}
Trofimov \cite{MR4057287} has provided a basic mechanism that produces automorphism-regular Cayley graphs. This can then be applied to groups that are finitely presented, not residually finite and admit Cayley graphs of large girth. A concrete example is the Baumslag-Solitar group $\Gamma={\rm BS}(m,n) =  \langle a,b \mid ab^ma^{-1}=b^n \rangle$ with the generating set $$S:= \{a,b,a^{-1},b^{-1},a^2,a^{-2},ab,b^{-1}a^{-1}, b^4,b^{-4}\}$$ with $n > m > 8$, see the discussion in \cite{MR4057287}. Since $\Gamma$ is not residually finite, there exists $g \in \Gamma$ with $g \neq 1$ but $g \in \Lambda$ for all finite index subgroups $\Lambda \leq \Gamma$. Now, if $r=d(e,g)$, then the ${\rm Cay}(\Gamma,S)$ cannot have a perfect finite $r$-local model. Indeed, by Theorem \ref{thm:main}, every such model is of the form ${\rm Sch}(\Lambda \backslash \Gamma,S)$ and $\Lambda g= \Lambda$.
\end{proof}

\begin{remark}
The results of Trofimov \cite{MR4057287} are slightly weaker since there it is assumed from the start that the approximation is by vertex-transitive graphs -- that corresponds to the case that $\Lambda$ is normal in our main result. See also the work of de la Salle and Tessera \cite{MR4072155} and the references therein for a more in-depth analysis of coverings of $r$-local models.
\end{remark}

\section*{Acknowledgments}
I thank Mikael de la Salle for remarks on a first version of this note. He pointed out that Theorem \ref{thm:main} follows from Theorem E and Proposition 5.1 in \cite{MR4072155}, which more generally implies that if a Cayley graph $G$ of a finitely presented groups has a countable automorphism group, then there is $r$ such that any perfect finite $r$-local  model appears as a Schreier graph of ${\rm Aut}(G)$.  Combining this with Corollary 1.3 of \cite{leemandelasalle}, one gets the following statement as a consequence:
Every finitely presented group that is not residually finite has a Cayley graph that provides a counterexample to Carmesin's question. See also Corollary K in \cite{MR4072155}, which answers Question \ref{carmesin} implicitly.
However, the setting in \cite{leemandelasalle,MR4072155 } makes all this necessarily considerably more complicated. We still think that our approach is valuable since it is simple and straight to the point.


\begin{bibdiv}
\begin{biblist}


\bib{MR3460154}{article}{
   author={Frisch, Joshua},
   author={Tamuz, Omer},
   title={Transitive graphs uniquely determined by their local structure},
   journal={Proc. Amer. Math. Soc.},
   volume={144},
   date={2016},
   number={5},
   pages={1913--1918},
}
\bib{leemandelasalle}{article}{
   author={Leemann, Paul-Henry},
   author={de la Salle, Mikael},
   title={Cayley graphs with few automorphisms: the case of infinite groups},
   journal={Annales Henri Lebesgue},
   volume={5},
   date={2022},
   pages={73-92},
}

\bib{MR4072155}{article}{
   author={de la Salle, Mikael},
   author={Tessera, Romain},
   title={Characterizing a vertex-transitive graph by a large ball},
   journal={J. Topol.},
   volume={12},
   date={2019},
   number={3},
   pages={705--743},
}

\bib{MR4057287}{article}{
   author={Trofimov, Vladimir I.},
   title={On the limits of vertex-symmetric graphs and their automorphisms},
   language={Russian, with English and Russian summaries},
   journal={Tr. Inst. Mat. Mekh.},
   volume={25},
   date={2019},
   number={4},
   pages={226--234},
   issn={0134-4889},
   translation={
      journal={Proc. Steklov Inst. Math.},
      volume={309},
      date={2020},
      number={1},
      pages={S167--S174},
      issn={0081-5438},
   },
}

\end{biblist}
\end{bibdiv}
\end{document}